\documentclass{amsart}
\usepackage{amsmath}
\usepackage{amsfonts}
\usepackage{amsthm}
\usepackage{amssymb}
\usepackage{enumerate}
\usepackage{cite}
\theoremstyle{plain}
\newtheorem{thm}{Theorem}[section]

\newtheorem{cor}[thm]{Corollary}
\newtheorem{lemma}[thm]{Lemma}
\theoremstyle{definition}
\newtheorem{defn}[thm]{Definition}
\theoremstyle{remark}
\newtheorem{rmk}[thm]{Remark}

\makeatletter

\newcommand{\Rmnum}[1]{\expandafter\@slowromancap\romannumeral #1@}
\makeatother

\thanks{The first author is supported by NSF of China No. 11901205.
The second author is supported by NSF of China No.11925107, No.11688101 and No.11671057.}
\begin{document}

\title{Carleson measures on convex domains}
\author{Haichou Li\textsuperscript{1} $\&$ Jinsong Liu\textsuperscript{2,3} $\&$ Hongyu Wang\textsuperscript{2,3}}
\address{$1.$ College of mathematics and informatics, South China Agricultural University, Guangzhou, 510640, China}
\address{$2.$ HLM, Academy of Mathematics and Systems Science,
Chinese Academy of Sciences, Beijing, 100190, China}
\address{$3.$ School of
Mathematical Sciences, University of Chinese Academy of Sciences,
Beijing, 100049, China }

\email{hcl2016@scau.edu.cn, liujsong@math.ac.cn, wanghongyu16@mails.ucas.ac.cn}

\begin{abstract}
Following M.Abate and A.Saracco's work on strongly pseudoconvex domains in $\mathbb{C}^n$,
we characterize Carleson measures of $A^2(D)$ in bounded convex domains with smooth boundary of finite type. We also give examples of Carleson measures with uniformly discrete (with respect to the Kobayashi distance) sequences.
\end{abstract}

\maketitle

\section{\noindent{{\bf Introduction}}}
Let $A$ be a Banach space of holomorphic functions on a domain $D \subset{\mathbb{C}}^{n}$, Assume that $A$ is contained in $L^{p}(D)$ for some $p>0$. Recall that a finite positive Borel measure $\mu$ on $D$ is a Carleson measure of $A$ if there exists a constant $C>0$ such that
$$
\int_{D}|f|^{p} d \mu \leqslant C\|f\|_{A}^{p}, \quad\quad \forall f \in A.
$$

Given $0<p<+\infty,$ the Bergman space $A^{p}(D)$ of $D$ is the Banach space of holomorphic $L^{p}$-functions on $D$. That is, $A^{p}(D)=L^{p}(D) \cap \mathcal{O}(D),$ endowed with the $L^{p}$-norm.

 Carleson measures for Bergman spaces were
first described by Hastings \cite{hastings}, and independently by Oleinik\cite{1978Embedding} and Pavlov -
Oleinik\cite{1974Embedding}. Subsequently, Cima - Wogen \cite{1982A} adapted the result to the unit ball $\mathbb{B}^n$ in $\mathbb{C}^n$ by using
spherical caps at the boundary. Then, as essentially
noticed by Luecking \cite{Luecking1983A} and explicitly stated by Duren - Weir \cite{Duren2007The}, it is possible to give
a characterization for the Carleson measures of Bergman spaces of $\mathbb{B}^n$ by using the balls for
the Bergman (or Kobayashi, or pseudohyperbolic) distance. In $1995$, Cima - Mercer \cite{Cima1995Composition} characterized Carleson measures of the Bergman spaces of strongly pseudoconvex domains. In particular, they proved that the class of Carleson measures of $A^{p}(D)$ does not depend on $p.$

A particularly important Bergman space is, of course, $A^{2}(D)$, where the Bergman kernel
lives. This suggests the question of whether it is possible to characterize Carleson measures
using the Bergman kernel. This has been done by Duren and Weir \cite{Duren2007The} for the unit ball, and by Abate and Saracco \cite{2011Carleson} for the strongly pseudoconvex domains; Our first
main result is a similar characterization of Carleson measures in convex domains with smooth boundary of finite type.

Let $K: D \times D \rightarrow \mathbb{C}$ be the Bergman kernel of a bounded convex domain $D \subset\mathbb{C}^{n}$ For any finite positive Borel measure $\mu$ on $D,$ the Berezin transform of $\mu$ is the function $B \mu: D \rightarrow \mathbb{R}$ given by
$$
B \mu(z)=\int_{D} \frac{|K(\zeta, z)|^{2}}{K(z, z)} d \mu(\zeta).
$$

Now we can prove the promised characterization for the Carleson measures of $A^{2}(D)$. We recall that $\nu$ denotes the Lebesgue volume measure of $\mathbb{R}^{2 n},$ normalized so that $\nu\left(B_{1}(0)\right)=1$. We show (see Theorem \ref{thm31}):
\begin{thm}
Let $\mu$ be a finite positive Borel measure on a bounded convex
domain $D\subset\mathbb{C}^n$ with smooth boundary of finite type. Then the following statements are equivalent:\\
(1). $\mu$ is a Carleson measure of $A^{2}(D)$; \\
(2). the Berezin transform of $\mu$ is bounded;\\
(3). there exists $r_0\in(0,1)$, such that for any $r\in(0, \:{r}_0)$ and $z\in D$, we have
$$
\mu\left(B_{D}\left(z, r\right)\right) \leqslant C_{r} \nu\left(B_{D}\left(z, r\right)\right)
$$
for some $C_r>0$.
\end{thm}
 To prove the above theorem, we need the construction of the polydisk and esti-
mates of the Bergman kernel in convex domains of finite type due to Chen\cite{chen} and McNeal\cite{1994Estimates}(need a little correction, see \cite{Nikolov2013On}) which is mainly used in the proof related to (3).
Next we are going to construct explicit examples of Carleson measures in convex domains with finite type boundary. As in the unit disc and the unit ball, an important family of examples is provided by uniformly discrete sequences. Let $(X, \:d)$ be a metric space. Then a sequence of points $\left\{x_{j}\right\} \subset X$ is uniformly discrete if there exists $\delta>0$ such that $d\left(x_{j}, x_{k}\right) \geqslant \delta$ for all $j \neq k $.

The following result is similar with the results proved about the unit ball \cite{Duren2007The,2006Uniformly} and the strongly pseudoconvex domains \cite{2011Carleson}.

\begin{thm} Let $D \subset \subset \mathbb{C}^{n}$ be a convex bounded domain with finite type boundary, and let $\Gamma=$ $\left\{z_{j}\right\}$ be a sequence in $D .$ Then $\Gamma$ is a finite union of uniformly discrete sequences (with respect to the Kobayashi distance) if and only if $\sum\limits_{z_{j} \in \Gamma} \prod\limits_{i=1}^{n} \sigma_{i}(z_j) \delta_{z_{j}}$
is a Carleson measure of
$A^{2}(D),$ where $\delta_{z_{j}}$ is the Dirac measure at $z_{j}$ and $\sigma_{i}(z_j)$ is the polyradii defined in Section 2.
\end{thm}

\bigskip
\section{\noindent{{\bf Preliminaries}}}
\subsection{Notation}\
(1) \:For $z\in\mathbb{C}^n$, let $|\cdot|$ and $d$ denote the standard Euclidean
norm, and let $|z_1-z_2 |$ or $d(z_1,z_2)$ denote the standard Euclidean
distance of $z_1, \:z_2\in \mathbb{C}^n$.

(2) \:Given an open set $D\subsetneq\mathbb{C}^n,\:x\in
D$, denote
$$\delta_{D}(x)=\inf\left\{d(x,\: \xi):\xi\in\partial
D\right\}.$$

(3) \:For any curve $\sigma: I \rightarrow \mathbb C^n$, we denote its
Euclidean length by $l_d(\sigma)$ and the Kobayashi length by $l_k(\sigma)$.

(4) \:$\nu$ will be the Lebesgue measure.

(5) \:For all real numbers $a,b$, we denote $a\vee b:= \max\{a, \: b\}$ and $a\wedge b:= \min\{a, \:b\}$.

\subsection{Convex domains with boundary of finite type }
 A point $p\in\partial D$ is of finite type (in the sense of D'Angelo) means that the maximum order of contact of one-dimensional complex analytic varieties with $\partial D$ at $p$, is bounded.

For a neighbourhood $U$ of $p$, let $r$ be a real-valued function such that $D\cap U=\{z\in U:r(z)<0\}$. By a rotation of the canonical coordinates we can arrange that the normal direction to $\partial D$ at $p$ is given by the $\Re z_{1}$-axis. Then by using the implicit function theorem, we obtain a local defining function of the form $r\left(z_{1}, \ldots, z_{n}\right)=\Re z_{1}-F\left(\Im z_{1}, \ldots, \Re z_{n}, \Im z_{n}\right)$,
where $F$ is a convex function. For $q \in U$ and $\epsilon>0$, we will consider the level sets
$$\partial D_{q, \epsilon}=\{z \in U ; r(z)=\epsilon+r(q)\}$$
which are also convex by the choice of $r$.

Then there exists small ball $V\subset U$ centered at $p$ such that: for every $q \in V$ and a sufficiently small $\epsilon>0$, we can assign coordinates $\left(z_{1}, \ldots, z_{n}\right)$, $z_{i}=x_{i}+i x_{n+i}$, centered at $q,$ obtained by translating and rotating the canonical coordinates, and numbers $\tau^{i}(q, \epsilon)$ which measure the distance from $q$ to $\partial D_{q, \epsilon}$ along the complex line determined by the $z_{i}$-axis.

First, choose $z_{1}$ so that $d\left(q, \: \partial D_{q, \epsilon}\right)$ is achieved along the positive $x_{1}$-axis. Let $q_{1,\epsilon}$ be the point in $\partial D_{q, \epsilon}$ such that $\tau^{1}(q,\epsilon)=|q-q_{1,\epsilon}|= d\left(q, \partial D_{q, \:\epsilon}\right)$, and let $e_1$ be the unit vector in the direction of $x_1$-axis. Next choose a
unit vector $e_2$ in the orthogonal complement of the space $e_1$ (the complex
linear span of $e_1$) such that the minimum distance from $q$ to $\partial D_{q, \epsilon}$ along
directions orthogonal to $e_1$ is achieved along the line given by $e_2$ in a
point $q_{2,\epsilon}$. Choose $z_2$ such that $x_2$-axis lies in the direction of $e_2$ and $\tau^{2}(q,\epsilon)=|q-q_{2,\epsilon}|$. Now continue by choosing a unit vector $e_3$ in the orthogonal complement of $e_1,e_2$, until the basis is complete.
 Also note that the remaining points $z_{i}, \: i=3, \ldots, n$, have the property that the distance from $q$ to $\partial D_{q, \epsilon}$ within the $z_{i}$-axis is achieved on the positive $x_{i}$-axis.

Therefore
$$
P(q,\epsilon)=\{|z_i|\leq\tau_{i}(q,\epsilon), \:i=1,\cdots, n\}
$$
is the corresponding polydisk constructed in terms of the minimal basis in $D\cap U$.

Suppose $D$ is a convex domain with smooth boundary of finite type. Then by the compactness of $\partial \Omega$, we can choose $\{(p_j, V_j, U_j), \:j=1\cdots K\}$ such that
$$\bigcup\limits_{j=1}^{K} V_j\supset\partial D,$$
where $V_j\subset U_j$ are both open neighbourhood of boundary points $p_j$ as above.
Take an open neighbourhood $W$ of $\partial D$ such that $$\bigcup\limits_{j=1}^{K} V_j\supset W\supset\partial D.$$
Then, for any $z_0\in W$, there exists at least one $j\in\mathbb{N}$ such that $z_0\in V_j$.
Denote by $P^{j}(z_0,\epsilon)$ the polydisk constructed in terms of the minimal basis in $D\cap U_j$.

Supposing $z_0\in V_j\cap V_k$ for $j\neq k$, by taking $\epsilon=|r(z_0)|$, then we have
\begin{align}\label{equal}
\tau^{j}_{i}(z_0,\epsilon)=\tau^{k}_{i}(z_0,\epsilon)
\end{align}
for $i=1,\cdots,n.$

By repeating the above construction, we can also get a global minimal basis in $D$.
For any $q\in D$, choose $q_1\in\partial D$ such that $\sigma_1(q):= |q - q_1|=\delta_{D}(q)$. Put $H_{1}=q+\operatorname{span}\left(q_{1}-q\right)^{\perp}$
and $D_{1}=D \cap H_{1}$. Let $q_{2} \in \partial D_{1}$ with $\sigma_{2}(q):=\left|q_{2}-q\right|=\delta_{D_{1}}(q) .$ Put
$H_{2}=q+\operatorname{span}\left(q_{1}-q, q_{2}-q\right)^{\perp}, \: D_{2}=D \cap H_{2}$, and so on. Thus we get an
orthonormal basis consisting of the vectors $\displaystyle{e_{i}=\frac{q_{i}-q}{\left\|q_{i}-q\right\|}}$. Finally, we choose $x_i$-axis in the direction of $e_i, 1 \leq i \leq n$.

Note that, since $D$ has smooth boundary, the numbers $\sigma_{i}(q)$ are uniquely determined when $q$ is near $\partial D$.
Actually by (\ref{equal}), we can assume that, for any $z_0\in W$,
\begin{align}\label{equal2}
\sigma_{i}(z_0)=\tau^{j}_{i}(z_0,\epsilon)
\end{align}
for some $j=1, \cdots, K$.

\subsection{The Kobayashi metric}

Given a domain $
D \subset \mathbb{C}^{n}\: (n\geq 2)$, the (infinitesimal)
Kobayashi metric is the pseudo-Finsler metric defined by
$$k_{
D}(x ; v)=\inf \left\{|\xi| : f \in \operatorname{Hol}(\mathbb{D}, D), \:\text { with } f(0)=x,
d(f)_{0}(\xi)=v\right\}.$$ Define the Kobayashi length of any curve $\sigma:[a,b]\rightarrow D$
to be
$$l_k(\sigma)=\int_{a}^{b} k_{D}\left(\sigma(t) ; \sigma^{\prime}(t)\right) d
t.$$
It is a consequence of a result due to Venturini \cite{VenturiniPseudodistances}, which is based on an observation by Royden \cite{royden1971remarks}, that
the Kobayashi pseudo-distance can be given by:
\begin{align*}
d_{K}(x, y)&=\inf_\sigma \big\{l_k(\sigma)| \:\sigma :[a, b]
\rightarrow D \text { is any absolutely continuous curve }\\
& \text { with } \sigma(a)=x \text { and } \sigma(b)=y \big\}.
\end{align*}
The main property of the Kobayashi pseudo - distance is that it is contracted by holomorphic naps: if $f: X \rightarrow Y$ is a holomorphic map, then
$$
\forall z, w \in X \quad d^{Y}_{K}(f(z), f(w)) \leqslant d^{X}_{K}(z, w).
$$
In particular, the Kobayashi distance is invariant under biholomorphisms, and decreases under inclusions: if $D_{1} \subset D_{2} \subset \subset \mathbb{C}^{n}$ are two bounded domains, then we have $d_{K}^{D_{2}}(z, w) \leqslant d_{K}^{D_{1}}(z, w)$ for all $z, w \in D_{1}$

If $X$ is a hyperbolic manifold, $z_{0} \in X$ and $r \in(0,1),$ then we shall denote by $B_{X}\left(z_{0}, r\right)$ the Kobayashi ball of centre $z_{0}$ and radius $(1 / 2) \log (1+r) /(1-r)$. That is,
$$
B_{X}\left(z_{0}, r\right)=\left\{z \in X \mid \tanh d_{K}^{X}\left(z_{0}, z\right)<r\right\}.
$$
Note that $\rho_{X}=\tanh K_{X}$ is still a distance on $X,$ because tanh is a strictly convex function on $\mathbb{R}^{+} .$

There are some estimates concerning the Kobayashi metric on convex domains.
\begin{lemma}\label{est}
Let $D\subset\subset\mathbb{C}^n$ be a convex domain with smooth boundary. Fix $\omega_0\in\Omega$ there exist $C_1,C_2>0$ such that, for any $ z, \: \omega\in D$,
\begin{align}
C_1-\frac{1}{2} \log \delta_{D}(z) \leqslant d_{K}\left(z_{0}, z\right) \leqslant C_2-\frac{1}{2} \log \delta_{D}(z).
\end{align}
\end{lemma}
\begin{lemma}
Let $D\subset\subset\mathbb{C}^n$ be a convex domain with smooth boundary of finite type. Then there exists $C_{3}>0$ such that, for every $z_0\in D$ and $r\in(0,1)$, we have the estimate
$$
\frac{C_{3}}{1-r}\delta_{D}(z_0)\geq\delta_{D}(z)\geq \frac{1-r}{C_{3}}\delta_{D}(z_0),
$$
for any $z\in B_{D}(z_0,r)$.
\end{lemma}
\begin{proof}
Fix $\omega_0\in D$. Then
\begin{align*}
C_1-\frac{1}{2} \log \delta_{D}(z) & \leqslant K_{D}\left(w_{0}, z\right) \leqslant K_{D}\left(z_{0}, z\right)+K_{D}\left(z_{0}, w_{0}\right) \\
& \leqslant \frac{1}{2} \log \frac{1+r}{1-r}+C_2-\frac{1}{2} \log \delta_{D}(z_0)
\end{align*}
for all $z\in B_{D}(z_0,r)$, and hence
$$
e^{2\left(C_1-C_{2}\right)} \delta_{D}(z_0) \leqslant \frac{2}{1-r} \delta_{D}(z).
$$
The left-hand inequality is obtained in the same way by reversing the roles of $z_0$ and $z$.
\end{proof}
\begin{rmk}\label{rmk}
Suppose $D=\{z:r(z)<0\}$ is bounded domain with smooth boundary, then $r(z)\approx\delta_{D}(z)$ uniformly in $z\in D$, thus after enlarging the constant $C_{3}$ if necessary, we also have for any $z_0\in D$ and $z\in B_{D}(z_0,r)$
$$
\frac{C_{3}}{1-r}r(z_0)\geq r(z)\geq \frac{1-r}{C_{3}}r(z_0).
$$
\end{rmk}

There is a precise description of the Kobayashi ball with respect to the minimal basis:
\begin{thm}[Theorem 1,\cite{Nikolov2015The}]\label{est5}
If $D$ is a bounded convex domain with smooth boundary of finite type, there exists $C
_{11}>0$ such that for any $r\in(0,1)$ and $z\in D$
$$
\frac{r}{n}\mathbb{D}^n(z,\sigma(z))\subset B_{D}(z,r)\subset\frac{2r}{1-r}\mathbb{D}^n(z,\sigma(z)),
$$
where $\mathbb{D}^n(z,\sigma(z))=\{\omega:|\omega_i-z_i|<\sigma_i(z)\}$ with $\sigma_i(z)$ defined in \text{(\ref{equal2})}.
\end{thm}

We also note that the polydiscs have the following properties:

\begin{lemma}[propsition 2.4,2.5,\cite{1994Estimates}]\label{double}
Suppose $z_0\in W$.
There is a constant $C_{4}>0$ such that if $z_0\in V_j$ for any $j=1,\cdots,K$, then
$$
\frac{1}{C_{4}}P^j(z_0, \epsilon)\subset P^j(z_0, 2\epsilon)\subset C_{4} P^j(z_0, \epsilon).
$$
Moreover, if $z\in P^{j}(z_0,\epsilon)$, then
$$
\frac{1}{C_{5}} P^{j}(z,\epsilon)\subset P^j(z_0,\epsilon)\subset C_{5} P^{j}(z,\epsilon).
$$
\end{lemma}
\begin{lemma}\label{discrete}
Let $D$ be a smoothly bounded convex domain of finite type in $\mathbb{C}^{n}$.
 Then, for every $r \in(0,1),$ there exist $M \in \mathbb{N}$ and a sequence of points $\left\{z_{k}\right\} \subset D$ such that $D=\bigcup\limits_{k=0}^{\infty} B_{D}\left(z_{k}, r\right)$ and no point of $D$ belongs to more than $M$ of the balls $B_{D}\left(z_{k}, R\right),$ where $R=(1 / 2)(1+r)$.
\end{lemma}
\begin{proof}
Let $\{B_j=B_{D}(z_j,\frac{r}{3})\}_{j\in\mathbb{N}}$ be a sequence of Kobayashi balls covering $D$. We can
extract a subsequence $\left\{\Delta_{k}=B_{D}\left(z_{k}, r / 3\right)\right\}_{k \in \mathrm{N}}$ of disjoint balls in the following way: set $\Delta_{1}=$ $B_{1}$. Suppose that we have already chosen $\Delta_{1}, \ldots, \Delta_{l}$. We define $\Delta_{l+1}$ as the first ball in the sequence $\left\{B_{j}\right\}$ which is disjoint from $\Delta_{1} \cup \ldots \cup \Delta_{l} .$ In particular, by construction every $B_{j}$ must intersect at least one $\Delta_{k}$.

We now claim that $\left\{B_{D}\left(z_{k}, r\right)\right\}_{k \in \mathrm{N}}$ is a covering of $D $. Indeed, let $z \in D $. Since $\left\{B_{j}\right\}_{j \in \mathrm{N}}$ is a covering of $D$, there is a $j_{0} \in \mathbb{N}$ such that $z \in B_{j_{0}}$. As remarked above, we get $k_{0} \in \mathbb{N}$ such that $B_{j_{0}} \cap \Delta_{k_{0}} \neq \emptyset$. By taking $w \in B_{j_{0}} \cap \Delta_{k_{0}}$, then we have
$$
\rho_{D}\left(z, z_{k_{0}}\right) \leqslant \rho_{D}(z, w)+\rho_{D}\left(w, z_{k_{0}}\right) \leqslant \frac{2}{3} r,
$$
and $z \in B_{D}\left(z_{k_{0}}, r\right)$.

To conclude the proof, we have to show that there is $m=m_{r} \in \mathbb{N}$ so that each point $z \in D$ belongs to at most $m$ of the balls $B_{D}\left(z_{k}, R\right) .$ Put $r_{1}=\frac{1}{3}\min \{r, 1-r\}$ and $R_{1}=\frac{1}{6}(5+r)$. Since $z \in B_{D}\left(z_{k}, R\right)$ is equivalent to $z_{k} \in B_{D}(z, R),$ we obtain that $z \in B_{D}\left(z_{k}, R\right)$ implies $B_{D}\left(z_{k}, r_{1}\right) \subset B_{D}\left(z, R_{1}\right)$. Therefore, noting that the balls $\{ B_{D}\left(z_{k}, r_{1}\right)\}$ are pairwise disjoint, we deduce that
\begin{equation}
\operatorname{card}\left\{k \in \mathbb{N} \mid z \in B_{D}\left(z_{k}, R\right)\right\} \leqslant \frac{\nu\left(B_{D}\left(z, R_{1}\right)\right)}{\nu\left(B_{D}\left(z_{k}, r_{1}\right)\right)}.
\end{equation}

{\bf Case} $(1)$. Assume that $z\in W$. Then, from Theorem \ref{est5}, it follows that
\begin{align*}
\operatorname{card}\left\{k \in \mathbb{N} \mid z \in B_{D}\left(z_{k}, R\right)\right\} &\leqslant \frac{2n}{1-r}\frac{\nu(\mathbb{D}^{n}(z,\sigma(z)))}{\nu(\mathbb{D}^{n}(z_k,\sigma(z_k)))}\\
&=\frac{2nR_1}{r_1(1-R_1)}\frac{\nu(P(z,|r(z)|))}{\nu(P(z_k,|r(z_k)|))}\\
&=\frac{2nR_1}{r_1(1-R_1)}\frac{\nu(P(z,|r(z)|))}{\nu(P(z,|r(z_k)|))}\frac{\nu(P(z,|r(z_k)|))}{\nu(P(z_k,|r(z_k)|))}.
\end{align*}
Noting that $z_k\in B_{D}(z,R)$, by Remark \ref{rmk} and Lemma \ref{double}, it follows that
$$
\operatorname{card}\left\{k \in \mathbb{N} \mid z \in B_{D}\left(z_{k}, R\right)\right\} \leqslant \frac{2nR_1}{r_1(1-R_1)}C_{5} C_{4}^{|\log\frac{1-R}{C_{3}}|}.
$$

{\bf Case} $(2)$. Assume that $z\in D\backslash W$. Since $D\backslash W$ is compact, there exists some $M>0$ such that
$$\operatorname{card}\left\{k \in \mathbb{N} \mid z \in B_{D}\left(z_{k}, R\right)\right\} \leqslant \frac{\nu\left(B_{D}\left(z, R_{1}\right)\right)}{\nu\left(B_{D}\left(z_{k}, r_{1}\right)\right)}\leqslant M.
$$
\end{proof}

\begin{lemma}\label{pluri}
Let $D \subset \subset \mathbb{C}^{n}$ be a bounded convex domain with smooth boundary of finite type. Then, for any $r\in(0,1)$ and $z_{0} \in W$, we have
$$
\varphi\left(z_{0}\right) \leqslant \frac{2n}{1-r}\frac{1}{\nu\left(B_{D}\left(z_{0}, r\right)\right)} \int_{B_{D}\left(z_{0}, r\right)} \varphi d \nu
$$
for all nonnegative plurisubharmonic functions $\varphi: D \rightarrow \mathbb{R}^{+}$.
\end{lemma}
\begin{proof}
By applying the sub-mean value property to each value separately,
we deduce that: on the polydisk $\mathbb{D}^n(z_0,\epsilon)=\{z:|z_i-z_{0,i}|<\epsilon_i\}\subset D$, it holds
$$
\varphi\left(z_{0}\right) \leqslant \frac{1}{\nu\left(\mathbb{D}^n(z_0,\epsilon)\right)} \int_{\mathbb{D}^n(z_0,\epsilon)} \varphi d \nu,
$$
as $\varphi$ is a plurisubharmonic function.

Theorem \ref{est5} implies that
$$\frac{r}{n}\mathbb{D}^n(z_0,\sigma(z_0))\subset B_{D}(z_0,r)\subset\frac{2r}{1-r}\mathbb{D}^n(z_0,\sigma(z_0)).$$
Thus,
\begin{align*}
\frac{1}{\nu\left(B_{D}\left(z_{0}, r\right)\right)} \int_{B_{D}\left(z_{0}, r\right)} \varphi d \nu&\geq \frac{1-r}{2r\nu\left(\mathbb{D}^n(z_0,\sigma(z_0))\right)} \int_{\frac{r}{n}\mathbb{D}^n(z_0,\sigma(z_0))} \varphi d \nu\\
&\geq\frac{1-r}{2n}\varphi(z_0),
\end{align*}
which completes the proof.
\end{proof}

\begin{cor}\label{sub}
Let $D \subset \subset \mathbb{C}^{n}$ be a bounded convex domain with smooth boundary of finite type. For any $r \in$ $(0,1)$, denote $R=(1 / 2)(1+r) \in(0,1)$. Then, for any $z_{0} \in W$ and $ z \in B_{D}\left(z_{0}, r\right)$, it holds
$$
\varphi(z) \leqslant \frac{8n^2r}{(1-r)^3}\frac{1}{\nu\left(B_{D}\left(z_{0}, r\right)\right)} \int_{B_{D}\left(z_{0}, R\right)} \varphi d \nu
$$
for every nonnegative plurisubharmonic function $\varphi: D \rightarrow \mathbb{R}^{+}$.
\end{cor}
\begin{proof}
Let $r_1 = \frac{1}{2}(1-r)$. By using the triangle inequality, $z\in B_D(z_0, r)$ yields $B_D(z, r_1)\subset B_D(z_0, R)$. Lemma \ref{pluri} and Theorem \ref{est5} then imply that
\begin{align*}
\varphi(z) & \leqslant \frac{2n}{1-r}\frac{1}{\nu\left(B_{D}\left(z, r_{1}\right)\right)} \int_{B_{D}\left(z, r_{1}\right)} \varphi d \nu\\& \leqslant \frac{2n}{1-r}\frac{1} {\nu\left(B_{D}\left(z, r_{1}\right)\right)} \int_{B_{D}\left(z_{0}, R\right)} \varphi d \nu \\
&=\frac{2n}{1-r} \frac{\nu\left(B_{D}\left(z_{0}, r\right)\right)}{\nu\left(B_{D}\left(z, r_{1}\right)\right)} \frac{1}{\nu\left(B_{D}\left(z_{0}, r\right)\right)} \int_{B_{D}\left(z_{0}, R\right)} \varphi d \nu\\
&\leqslant \frac{8n^2r}{(1-r)^3}\frac{1}{\nu\left(B_{D}\left(z_{0}, r\right)\right)}\int_{B_{D}\left(z_{0}, R\right)} \varphi d \nu,
\end{align*}
for all $z\in B_{D}(z_0,\: r)$. The proof is complete.
\end{proof}

\subsection{The Bergman kernel}
Let $K: D \times D \rightarrow \mathbb{C}$ be the Bergman kernel of the domain $D$. It has the reproducing property; that is, for any $z \in D$,
$$
 f(z)=\int_{D} K(z, \zeta) f(\zeta) d \nu, \quad \quad \forall f \in A^{2}(D).
$$
Since $K(\cdot, \zeta)=\overline{K(\zeta, \cdot)} \in A^{2}(D)$, particularly we have
$$
K(z,\: z)=\int_{D}|K(z, \:\zeta)|^{2} d \nu(\zeta)=\|K(z, \:\cdot)\|_{2}^{2}.
$$
For each $z_{0} \in D$, let $k_{z_{0}} \in A^{2}(D)$ denote the normalized Bergman kernel given by
$$
k_{z_{0}}(z)=\frac{K\left(z, z_{0}\right)}{\left\|K\left(\cdot, z_{0}\right)\right\|_{2}}=\frac{K\left(z, z_{0}\right)}{\sqrt{K\left(z_{0}, z_{0}\right)}}.
$$
Clearly, $\left\|k_{z_{0}}\right\|_{2}=1$. The Berezin transform $B \mu$ of a finite measure $\mu$ on the domain $D$ is the function given by
$$
B \mu(z)=\int_{D}\left|k_{z}(\zeta)\right|^{2} d \mu(\zeta)
$$
for all $z \in D$
\begin{lemma}[Theorem 3.4,\cite{1994Estimates}]\label{kernel}
Let $D \subset \subset \mathbb{C}^{n}$ be a bounded convex domain with smooth boundary of finite type, and let $V_{j}\subset U_{j}$ be the neighborhood of $p_{j}\in\partial D$ in Section 2.2. Then if $z\in V_{j}\cap D$,
\begin{align}
K_{D}(z,z)\gtrsim\prod\limits_{i=1}^{n}\tau^{j}_i(z,|r(z)|)^{-2}.
\end{align}
\end{lemma}
Note that we can take a constant $C_{6}$ uniformly in $j$, such that for any $z\in W\subset\bigcap\limits_{j=1}^{K} V_j$,
 $$
 K_{D}(z,z)\geq C_{6}\prod\limits_{i=1}^{n}\sigma_{i}(z)^{-2}.
 $$

\begin{lemma}[Theore 5.2, \cite{1994Estimates}]\label{est2}
Let $D \subset \subset \mathbb{C}^{n}$ be a bounded convex domain with smooth boundary of finite type. Suppose that $V_{j}\subset U_{j}$ is the neighborhood of $p_{j}\in\partial D$. Then, for all multi-indices $\mu, \:\nu$, there exists a constant $C_{j,\mu,\nu}$ such that, for all $z, \:\omega\in V_{j}\cap D$,
\begin{align*}
|D^{\mu}\bar{D}^{\nu}K_{D}(z,\omega)|\leq C_{j,\mu,\nu}\prod\limits_{i=1}^{n}\tau_i^{j}(z,|r(z)|+|r(\omega)+M^j(z,\omega)|)^{-2-\mu_i-\nu_i},
\end{align*}
where $M^{j}(z,\omega)=\inf\{\epsilon>0:\omega\in P^{j}_{\epsilon}(z)\}.$
\end{lemma}
First suppose $\omega\in B_{D}(z,r)$ for some $r<\frac{1}{3}$. Shrinking $V_j$ as necessary, we may assume that $\omega\in V_j$. Then by Theorem \ref{est5}, it follows that
$$B_{D}(z,r)\subset \frac{2r}{1-r}\mathbb{D}^n(z,\sigma(z)) \subset P^j_{|r(z)|}(z).$$
Therefore
 $$M^{j}(z,\omega)\leq|r(z)|.
 $$
By using Remark \ref{rmk}, we deduce that
$$
|r(z)|\leq |r(z)|+|r(\omega)|+M^{j}(z,\omega)\leq (2+\frac{2}{3}C_{3})|r(z)|.
$$
Then Lemma \ref{double} implies that,
\begin{align}\label{est9}
|D^{\mu}\bar{D}^{\nu}K_{D}(z,\omega)|&\leq C_{j,\mu,\nu}\prod\limits_{i=1}^{n}\tau_i^{j}(z,|r(z)|+|r(\omega)+M^j(z,\omega)|)^{-2-\mu_i-\nu_i}\nonumber\\
&\leq C_{4}^{\log_2(2+\frac{2}{3}C_{3})}C_{j,\mu,\nu}\prod\limits_{i=1}^{n}\tau_i^{j}(z,|r(z)|)^{-2-\mu_i-\nu_i}\nonumber\\
&\leq C_{\mu,\nu}\prod\limits_{i=1}^{n}\sigma_{i}(z)^{-2-\mu_i-\nu_i},
\end{align}
where $\omega\in B_{D}(z,r)$ and $C_{\mu,\nu}=\max\left\{C_{4}^{\log_2(2+\frac{2}{3}C_{3})}C_{j,\mu,\nu}: \:j=1\cdots K\right\}$.
\begin{lemma}\label{kernel2}
Let $D \subset \subset \mathbb{C}^{n}$ be a bounded convex domain with smooth boundary of finite type. There exists $r_0$ such that, $\forall z_0\in W, \:\: r\in(0,r_0)$ and $\omega\in B_{D}(z_0,\: r)$,
\begin{align}
K_{D}(z,\omega)\geq \frac{C_{6}}{2} \prod\limits_{i=1}^{n}\sigma_i(z)^{-2}.
\end{align}
\end{lemma}
\begin{proof}
By Theorem \ref{est5}, we have
\begin{align}\label{est10}\frac{\left|(\omega_{i}-z_i)\right|}{\sigma_{i}(z)}\leq \frac{2r}{1-r},\quad i=1\cdots n.\end{align}
%
Denote $$K_{D}(z,\cdot)=f(\cdot)+ig(\cdot)$$ where $f$ and $g$ are real functions on convex domains $D$. By applying the Lagrange's mean value Theorem to $f$ and $g$ respectively, there exist $\theta_1,\: \theta_2\in[0,1]$ such that if we write $z_k=x_k+ix_{n+k}$ and $\omega_k=\xi_k+i\xi_{n+k}$, then
 $$f(\omega)-f(z)=\sum\limits_{k=1}^{2n}\frac{\partial f}{\partial x_k}(z+\theta_1(\omega-z))(\xi_{k}-x_{k})$$
 and $$g(\omega)-g(z)=\sum\limits_{k=1}^{2n}\frac{\partial g}{\partial x_k}(z+\theta_2(\omega-z))(\xi_{k}-x_{k}).
 $$
 Noting that, for $f$ (or $g$), we have
 $$
 \left|\frac{\partial f}{\partial x_k}\right|+\left|\frac{\partial f}{\partial x_{n+k}}\right|\leq4\left|\frac{\partial f}{\partial \bar{z}_k}\right|, \quad \left|\frac{\partial g}{\partial x_k}\right|+\left|\frac{\partial g}{\partial x_{n+k}}\right|\leq4\left|\frac{\partial g}{\partial \bar{z}_k}\right|.
 $$
 By denoting $C=\max\{C_{\mu,\nu}:|\mu|=0,|\nu|=1\}$ and using (\ref{est9}) and (\ref{est10}), we have
\begin{align*}
 |K_{D}(z,\omega)-K_{D}(z,z)| &\leq |f(\omega)-f(z)|+|g(\omega)-g(z)|\\
 &\leq8 C\sum\limits_{k=1}^n \prod\limits_{i=1}^{n}\sigma_i(z)^{-2}\sigma_k(z)^{-1}\frac{2r}{1-r}\sigma_k(z)\\
 &\leq C\frac{16nr}{1-r} \prod\limits_{i=1}^{n}\sigma_i(z)^{-2}.
\end{align*}
Therefore, if we choose $r_0$ small enough such that $$ C\frac{16nr_0}{1-r_0}<\frac{C_{6}}{2},$$
then for any $r\in(0, \:r_0)$ and $\omega\in B_{D}(z,r)$, we have the estimate
\begin{align*}
K_{D}(z,\omega)&\geq K_{D}(z,z)-|K_{D}(z,z)-K_{D}(z,\omega)|\\
&\geq \frac{C_{6}}{2} \prod\limits_{i=1}^{n}\sigma_i(z)^{-2}
\end{align*}
\end{proof}

By Lemma \ref{kernel} and Lemma \ref{kernel2}, the following corollary is direct:
\begin{cor}\label{est4}
Let $D \subset \subset \mathbb{C}^{n}$ be a bounded convex domain with smooth boundary of finite type.
There exists $r_0 \in(0,1)$ and $C_{7}>0$ such that if $z_{0} \in W, \:\forall r\in(0, \:r_0), \: z \in B_{D}\left(z_{0}, r\right)$, then
$$
\left|k_{z_{0}}(z)\right|^{2} \geqslant C_{7}\prod\limits_{i=1}^{n}\sigma_i(z_0)^{-2}.
$$
\end{cor}
\begin{proof}
By taking $\mu=\nu=0$ in Lemma \ref{est2}, the inequality follows by the above Lemma \ref{kernel2}.
\end{proof}

\bigskip
\section{Carleson measure}
Let $D \subset \subset \mathbb{C}^{n}$ be a convex domain with smooth boundary of finite type. The Bergman space $A^{2}(D)$ of $D$ is the Banach space of holomorphic $L^{2}$-functions on $D,$ that is, $A^{2}(D)=L^{2}(D) \cap \mathcal{O}(D),$ endowed with the $L^{2}$-norm
$$
\|f\|_{2}^{2}=\int_{D}|f(z)|^{p} d\nu,
$$
where $\nu$ is the Lebesgue measure normalized with $\nu\left(B_{1}(0)\right)=1$.
A finite positive Borel measure $\mu$ on $D$ is said to be a Carleson measure of $A^{2}(D)$ if there exists $C>0$ such that
$$
\forall f \in A^{2}(D) \quad \int_{D}|f(z)|^{2} d \mu \leqslant C\|f\|_{2}^{2}.
$$

\begin{thm}\label{thm31}

Let $\mu$ be a finite positive Borel measure on a bounded convex
domain $D\subset\mathbb{C}^n$ with smooth boundary of finite type. Then the following statements are equivalent:\\
$(1)$ $\mu$ is a Carleson measure of $A^{2}(D)$; \\
$(2)$ the Berezin transform of $\mu$ is bounded;\\
$(3)$ for some(and hence for any) $r\in(0,r_0)$, there exists $C_r>0$ such that $\mu\left(B_{D}\left(z, r\right)\right) \leqslant C_{r} \nu\left(B_{D}\left(z, r\right)\right)$ for all $z\in D$.
\end{thm}
The proof of Theorem \ref{thm31} follows from Corollary \ref{sub}, Theorem \ref{est5} and Corollary \ref{est4}.
\begin{proof}
(1)$\Rightarrow$(2): Take $r_0$ as in Lemma \ref{kernel2}. Then
$$B \mu\left(z_{0}\right)=\int_{D}\left|k_{z_{0}}(z)\right|^{2} d \mu(z) \leqslant C\left\|k_{z_{0}}\right\|_{2}^{2}=C.$$
(2)$\Rightarrow$(3):
If $z_0\notin W$, then we have
$$\mu\left(B_{D}\left(z_{0}, r\right)\right) \leq \mu(D) \leq \frac{n\mu({D})}{r\sigma(z_0)^2} \nu\left(B_{D}\left(z_{0}, r\right)\right)\leq\frac{n\mu({D})}{r\delta_{D}(z_0)^{2n}} \nu\left(B_{D}\left(z_{0}, r\right)\right).
$$
If $z_0\in V_j$ for some $j\in \{1, 2, \cdots, K\}$,
noting that the Berezin transform is bounded, then there exists $C > 0$
independent of $z_0$ and $r$ such that
$$
\int_{B_{D}\left(z_{0}, r\right)}\left|k_{z_{0}}(z)\right|^{2} d \mu(z) \leqslant B \mu\left(z_{0}\right) \leqslant C.
$$
Then by Corollary \ref{est4}
$$
C_{7}\prod\limits_{i=1}^{n}\sigma_i(z_0)^{-2}\mu(B_{D}(z_0,r))\leq \int_{B_{D}\left(z_{0}, r\right)}\left|k_{D}(z,z_0)\right|^{2} d \mu(z)\leqslant C.
$$
Thus by Theorem \ref{est5} we obtain
$$
\mu(B_{D}(z_0,r))\leq \frac{C}{C_{7}}\prod\limits_{i=1}^{n}\sigma_i(z_0)^{2}\leq \frac{Cn}{C_{7}r}\nu(B_{D}(z_0,r)).
$$
(3)$\Rightarrow$(1):
 Let $\{z_k\}$ be the sequence given by Lemma \ref{discrete}. Clearly, for all $f\in A^{2}(D)$, we have
 $$
 \int_{D}|f(z)|^{2} d \mu(z) \leqslant \sum_{k=1}^{\infty} \int_{B_{D}\left(z_{k}, r\right)}|f(z)|^{2} d \mu(z).
 $$
Since $|f|^2$ is nonnegative plurisubharmonic functions, by Lemma \ref{sub},
 $$\begin{aligned}
\int_{B_{D}\left(z_{k}, r\right)}|f(z)|^{2} d \mu(z) & \leqslant \frac{K_{r}}{\nu\left(B_{D}\left(z_{k}, r\right)\right)} \int_{B_{D}\left(z_{k}, r\right)} d \mu(z) \int_{B_{D}\left(z_{k}, R\right)}|f(\zeta)|^{2} d \nu(\zeta) \\
&=K_{r} \frac{\mu\left(B_{D}\left(z_{k}, r\right)\right)}{\nu\left(B_{D}\left(z_{k}, r\right)\right)} \int_{B_{D}\left(z_{k}, R\right)}|f(\zeta)|^{2} d \nu(\zeta) \\
& \leqslant K_{r} C_{r} \int_{B_{D}\left(z_{k}, R\right)}|f(\zeta)|^{2} d \nu(\zeta),
\end{aligned}
$$
where $R=(1 / 2)(1+r)$. Hence,
$$
\int_{D}|f(z)|^{2} d \mu(z) \leqslant K_{r} C_{r} \sum_{k=1}^{\infty} \int_{B_{D}\left(z_{k}, R\right)}|f(\zeta)|^{2} d \nu(\zeta) \leqslant K_{r} C_{r} M\|f\|_{2}^{2},
$$
as desired.
\end{proof}

%
\bigskip
\section{Uniformly sequences}
\begin{defn}[Uniformly discrete sequence]
Suppose $(X, d)$ is a metric space. A sequence $\Gamma=\left\{x_{i}\right\} \subset X$ of points  is uniformly discrete if there exists $\delta>0$ s.t $d\left(x_{i}, x_{j}\right) \geqslant \delta$ for all $i \neq j .$  And inf $_{i \neq j} d\left(x_{i}, x_{j}\right)$ is called the separation constant of $\Gamma$.

Besides, given $x_{0} \in X, r>0,$ and a subset $\Gamma \subset X,$ we write the number of points of $\Gamma$ contained in the ball centered at $x_{0}$ and radius $r$ by $M\left(x_{0}, r, \Gamma\right)$.
\end{defn}
 The following result is due to M.Abate and A.Saracco, we add their proof here for the sake of completeness.
\begin{lemma}[\cite{2011Carleson}, Lemma 4.1]\label{discrete2}
 Suppose $X$ be a metric space, and $\Gamma=\left\{x_{n}\right\}_{n \in \mathbf{N}} \subset X$ be a sequence in $X .$ If there exists $ M \geqslant 1$ and $r>0$ such that $M(x, r, \Gamma) \leqslant M$ for all $x \in X,$ then $\Gamma$ is the union of at most $M$ uniformly discrete sequences.
 \end{lemma}
 \begin{proof}
We want to prove that $\Gamma=\Gamma_{0} \cup \ldots \cup \Gamma_{M-1} .$ where $\Gamma_{i}$ are uniformly discrete sequences.

First of all, suppose that $\Gamma_{0}=\ldots=\Gamma_{M-1}=\emptyset$ . By induction
on $n\geq 0,$ we  put each $x_{n}$ in a specified $\Gamma_{j} .$ And we write the metric ball with center $x$ and radius $r$ by $B(x, r)$. Besieds, we write $m\left(x_{i}\right)=j$ if $x_{i} \in \Gamma_{j},$.

Futhermore, we can set $x_{0} \in \Gamma_{0}$. If we have already defined $m\left(x_{i}\right)$ for $i \leqslant n,$ then consider $x_{n+1}$. Therefore, $\Gamma \cap B\left(x_{n+1}, r\right)$ contains no more than $M$ points, and one of them is $x_{n+1}$. So, $\left\{x_{0}, \ldots, x_{n}\right\} \cap B\left(x_{n+1}, r\right)$ contains no more than $M-1$ points.

 Thus we can write
$m\left(x_{n+1}\right)=\min \{i \in\{0, \ldots, M-1\} \mid i \neq m\left(x_{j}\right)\}$ for all $0 \leqslant j \leqslant n$ such that $\left.x_{j} \in B\left(x_{n+1}, r\right)\right\}$.
As a result, $d\left(x_{n+1}, x_{j}\right) \geqslant r$ for all $x_{j} \in \Gamma_{m\left(x_{n+1}\right)}$ with $j<n+1$. Finally, by construction, if $x_{h}, x_{k} \in \Gamma_{j}$ with $h>k,$ then we have $d\left(x_{h}, x_{k}\right) \geqslant r$. Therefore $\Gamma_{0}, \ldots, \Gamma_{M-1}$ are uniformly discrete sequences and we complete the proof.
 \end{proof}
 \begin{thm} Let $D \subset \subset \mathbf{C}^{n}$ be a convex bounded domain with boundary of finite type, and let $\Gamma=$ $\left\{z_{k}\right\}$ be a sequence in $D .$ Then the following statements are equivalent:\\
$(1)$ $\Gamma$ is a finite union of uniformly discrete $($with respect to the Kobayashi distance$)$ sequences.\\
$(2)$$\sum\limits_{z_{k} \in \Gamma} \prod\limits_{i=1}^{n} \sigma_{i}(z_k) \delta_{z_{k}}$ is a Carleson measure of $A^{2}(D).$ \\
$(3)$$\sum\limits_{z_{k} \in \Gamma} \prod\limits_{i=1}^{n} \sigma_{i}(z_k)|f(z_k)|^2\leq C\|f\|_{2}^2.$\\
where $\delta_{z_{k}}$ is the Dirac measure at $z_{k}$.
\end{thm}
\begin{proof}
$(1)\Rightarrow(3):$ It suffices to prove the assertion when $\Gamma$ is a single uniformly discrete sequence. Let $\delta>0$ be the separation constant of $\Gamma,$ and put $r=\delta / 2 \wedge r_0.$ By the triangle inequality, the Kobayashi balls $B_{D}\left(z_{k}, r\right)$ are pairwise disjoint. Hence,
$$\int_{D}|f(z)|^{p} d \nu \geqslant \sum_{z_{k} \in \Gamma} \int_{B_{D}\left(z_{k}, r\right)}|f(z)|^{p} d \nu$$
Now, $|f|^p$ is plurisubharmonic and nonnegative; By Lemma \ref{pluri}
we have

$$|f(z_k)|^2 \leqslant \frac{2n}{1-r}\frac{1}{\nu\left(B_{D}\left(z_{k}, r\right)\right)} \int_{B_{D}\left(z_{k}, r\right)} |f|^2 d V$$
Then by Theorem \ref{est5}, we have
$$\sum\limits_{z_{k} \in \Gamma} \prod\limits_{i=1}^{n} \sigma_{i}(z_k)|f(z_k)|^2\leq \frac{2n^2}{r(1-r)}\|f\|_{2}^2$$
$(3)\Rightarrow(2):$ The proof is obvious.\\
$(2)\Rightarrow(1):$ By Corollary \ref{est4}, for any $z_0\in W,$
$$
\left|k_{z_{0}}(z)\right|^{2} \geqslant C_{7}\prod\limits_{i=1}^{n}\sigma_i(z_0)^{-2}
$$
Hence,
\begin{align*}
M(z_0, r, \Gamma)&\leq \frac{1}{C_{7}}\sum\limits_{z\in B_{D}(z_0,r)\cap \Gamma}\left|k_{z_{0}}(z)\right|^{2}\prod\limits_{i=1}^{n}\sigma_i(z_0)^{2}\\
&\leq \frac{C}{C_{7}}\|k_{z_{0}}(z)\|^2_2=\frac{C}{C_{7}}
\end{align*}
for some $C>0.$

For $z\in D/W,$ it's easy to see that $M(z_0, r, \Gamma)<\infty$.
Thus, by Lemma \ref{discrete2} we complete the proof.
\end{proof}

\bigskip
\bibliography{reference}
\bibliographystyle{plain}{}
\end{document}